\documentclass{article}

\usepackage{arxiv}

\usepackage[utf8]{inputenc} 
\usepackage[T1]{fontenc}    
\usepackage{hyperref}       
\usepackage{url}            
\usepackage{booktabs}       
\usepackage{amsfonts}       
\usepackage{nicefrac}       
\usepackage{microtype}      
\usepackage{lipsum}		
\usepackage{graphicx}
\usepackage{natbib}
\usepackage{doi}

\usepackage{amsmath}
\usepackage{amssymb}
\usepackage{amsfonts}
\usepackage{bm}

\usepackage{amsthm}
\usepackage{overpic}
\usepackage{lipsum}
\usepackage{enumitem}
\usepackage{cases}
\usepackage{subcaption}
\usepackage{graphicx}

\DeclareMathOperator*{\argmax}{arg\,max}
\DeclareMathOperator*{\argmin}{arg\,min}

\usepackage{ifthen}
\usepackage{graphicx,epstopdf,amssymb,amsthm}
\usepackage{ulem,color}
\usepackage{epstopdf}

\usepackage[utf8]{inputenc}
\usepackage[english]{babel}
\usepackage{caption}
\usepackage{algorithm,algorithmicx}
\usepackage{algpseudocode}

\DeclareCaptionLabelFormat{lc}{\MakeLowercase{#1}~#2}
{
  \theoremstyle{plain}
  \newtheorem{assumption}{Assumption}
}

\newtheorem{theorem}{Theorem}

\newtheorem{lemma}{Lemma}

\newtheorem{definition}{Definition}
\newtheorem{proposition}{Proposition}
\newtheorem{example}{Example}[]

\title{Collaborative Safe Formation Control for Coupled Multi-Agent Systems}

\date{} 					

\author{ \href{https://orcid.org/0000-0002-2489-4411}{\includegraphics[scale=0.06]{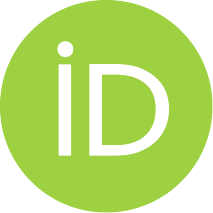}\hspace{1mm}Brooks A. Butler}\\
	Elmore Family School of Electrical and Computer Engineering\\
	Purdue University\\
	\texttt{brooksbutler@purdue.edu} \\
        \And
        Chi Ho Leung\\
	Elmore Family School of Electrical and Computer Engineering\\
	Purdue University\\
	\texttt{leung61@purdue.edu} \\
        \And
	\href{https://orcid.org/0000-0002-4095-7320}{\includegraphics[scale=0.06]{orcid.pdf}\hspace{1mm}Philip E. Par\'{e}}
        \thanks{This work was partially funded by Purdue’s Elmore Center for Uncrewed Aircraft Systems and the National Science Foundation,
        grant NSF-ECCS \#2238388.}\\
	Elmore Family School of Electrical and Computer Engineering\\
	Purdue University\\
	\texttt{philpare@purdue.edu} \\
}

\renewcommand{\headeright}{}
\renewcommand{\undertitle}{}
\renewcommand{\shorttitle}{Collaborative Safe Formation Control}

\hypersetup{
pdftitle={Collaborative Safe Formation Control for Coupled Multi-Agent Systems},
pdfsubject={q-bio.NC, q-bio.QM},
pdfauthor={Brooks A. Butler, Chi Ho Leung, Philip E. Par\'{e}},
pdfkeywords={Network analysis and control, Safety-critical control, Cooperative control, Constrained control},
}

\begin{document}
\maketitle

\begin{abstract}
The safe control of multi-robot swarms is a challenging and active field of research, where common goals include maintaining group cohesion while simultaneously avoiding obstacles and inter-agent collision. Building off our previously developed theory for distributed collaborative safety-critical control for networked dynamic systems, we propose a distributed algorithm for the formation control of robot swarms given individual agent dynamics, induced formation dynamics, and local neighborhood position and velocity information within a defined sensing radius for each agent. Individual safety guarantees for each agent are obtained using rounds of communication between neighbors to restrict unsafe control actions among cooperating agents through safety conditions derived from high-order control barrier functions. We provide conditions under which a swarm is guaranteed to achieve collective safety with respect to multiple obstacles using a modified collaborative safety algorithm. We demonstrate the performance of our distributed algorithm via simulation in a simplified physics-based environment.
\end{abstract}

\section{Introduction}

Nature has always inspired scientists and engineers to design elegant solutions for real-life problems.
One of the nature-inspired ideas in the field of automatic control comes from the observation that collective behavior in nature is often governed by relatively simple interactions among individuals (\cite{strogatz2004sync}).
The set of collaboration rules introduced by \cite{reynolds1987flocks} is one of the early attempts in the literature to describe collective formation behavior in the animal kingdom.
In more recent years, multi-agent formation problem has received special attention in robotics and automatic control due to its broad range of applications and theoretical challenges.
While it is impossible to exhaustively categorize every formation control-related research, we can organize them in terms of the fundamental ideas behind the control schemes (\cite{beard2001coordination, reynolds1987flocks}), sensing capability and interaction topology of the formation controller (\cite{oh2015survey}), and the formation control-induced problems of interest such as the consensus reaching problem (\cite{ren2010distributed}).

Some generalizations of the formation control-induced problems also find their application in other multi-agent cyber-physical systems outside the robotic community.
Some examples of critical multi-agent model applications include the mitigation of epidemic-spreading processes (\cite{pare2020modeling, butler2023optimal}), smart grid management (\cite{tuballa2016review}), and uncrewed aerial drone swarms (\cite{tahir2019swarms}).
Since many of these multi-agent cyber-physical systems have become ubiquitous in modern society, effective and safe operation in multi-agent systems is crucial, as disruptions in these interconnected systems can potentially have far-reaching societal and economic consequences.

Theoretical frameworks and techniques from the study of safety-critical control are natural solutions to the problem of collaborating safety requirements in the multi-agent formation problem.
Foundational work on safety critical control can be traced back to the 1940s (\cite{nagumo1942lage, blanchini1999set}).
Recently, the introduction and refinement of control barrier functions (CBFs) (\cite{ames2016control, ames2019control}) has induced new excitement in the field of safety-critical control.
Since their introduction, control barrier
functions have been used in numerous applications to provide safety guarantees in various dynamic system models (\cite{chung2018survey, wang2017safety}).
Moreover, multiple recent studies have reported CBFs' practicality and theoretical soundness in solving the  multi-agent obstacle avoidance problem (\cite{wang2017safety, santillo2021collision, jankovic2021collision}).


In this paper, we extend the work in \cite{butler2023distributed} to design a non-intrusive collaborative safety filter for formation control with online obstacle avoidance guarantees. The problem formulation and analysis are performed under the formality of CBFs.
The collaborative safety filter is realized by 
a novel communication algorithm wherein agents share their maximum safety capability within their neighborhood in the formation.
The maximum safety capability is computed from each agents' local distance-based sensor data, and therefore, is flexible for a wide range of real-life implementation scenarios.
We show in simulation that rounds of communication between agents terminate in finite time with consensus on the desired collaboratively safe control action if the underlying centralized constraint optimization problem is feasible.
The proofs for all lemmas and theorems presented in this work can be found in the full version of this paper at \cite{butler2023collaborative}.

We organize the remainder of this paper as follows. We introduce some preliminaries for networked models and safety-critical control in Section~\ref{sec:prelim}, then formally define the safe formation control problem in Section~\ref{sec:safe_form_prob}. We then present our proposed method for safe formation control through active collaboration via communication in Section~\ref{sec:safe_with_colab} and illustrate these results with a simplified two-dimensional formation control example in Section~\ref{sec:application}. 

\subsection{Notation}

Let $|\mathcal{C}|$ denote the cardinality of the set $\mathcal{C}$. $\mathbb{R}$ and $\mathbb{N}$ are the set of real numbers and positive integers, respectively. Let $C^r$ denote the set of functions $r$-times continuously differentiable in all arguments. We define ${\Vert \cdot \Vert}_2$ and ${\Vert \cdot \Vert}_{\infty}$ to be the two-norm and infinity norm of a given vector argument, respectively. We notate $\mathbf{0}$ and $\mathbf{1}$ to be vectors of all zeros and all ones, respectively, of the appropriate size given by context and $[v]_k$ to be the $k$th element of vector $v$. A monotonically increasing continuous function $\alpha: \mathbb{R}_{+} \rightarrow \mathbb{R}_{+}$ with $\alpha(0) = 0$ is termed as class-$\mathcal{K}$. We define $[n] \subset \mathbb{N}$ to be a set of indices $\{1, 2, \dots, n\}$.
We define the Lie derivative of the function $h:\mathbb{R}^N \rightarrow \mathbb{R}$ with respect to the vector field generated by $f:\mathbb{R}^N \rightarrow \mathbb{R}^N$ as
\begin{equation}
    \mathcal{L}_f h(x) = \frac{\partial h(x)}{\partial x} f(x).
\end{equation}
We define higher-order Lie derivatives with respect to the same vector field $f$ with a recursive formula (\cite{robenack2008computation}), where $k>1$, as
\begin{equation}
    \mathcal{L}^k_f h(x) = \frac{\partial \mathcal{L}^{k-1}_f h(x)}{\partial x} f(x).
\end{equation}
We compute the Lie derivative of $h$ along the vector field generated by $f$ and then along the vector field generated by $g$ as
\begin{equation}
    \mathcal{L}_g \mathcal{L}_f h(x) = \frac{\partial}{\partial x}\left(\frac{\partial h(x)}{\partial x} f(x)\right)g(x).
\end{equation}

\section{Preliminaries} \label{sec:prelim}

We define a networked system using a graph $\mathcal{G} = (\mathcal{V}, \mathcal{E})$, where $\mathcal{V}$ is the set of $n = \vert \mathcal{V} \vert$ nodes, $ \mathcal{E} \subseteq \mathcal{V}\times \mathcal{V} $ is the set of edges. Let $\mathcal{N}_i$ be the set of all neighbors with an edge connection to node $i \in [n]$, where 
\begin{equation}
    \mathcal{N}_i = \{j \in [n]\setminus \{i\}: (i,j) \in \mathcal{E} \}.
\end{equation}
We further define $x_i$ to be the state vector for agent $i \in  [n]$, $x_{\mathcal{N}_i}$ to be the concatenated states of all neighbors to agent $i$, i.e. $x_{\mathcal{N}_i} = \left( x_j, \forall j \in \mathcal{N}_i\right)$, and $x$ to be the full state of the networked system.

Recall the definition of \textit{high-order barrier functions} (HOBF) (\cite{xiao2019control,xiao2021high}), where we define a series of functions in the following form
\begin{equation} \label{eq:HO_funcs}
    \begin{aligned}
        \psi_i^0(x) &:= h_i(x) \\
        \psi_i^1(x) &:= \dot{\psi}_i^0(x) + \alpha_i^1(\psi_i^0(x)) \\
        & \vdots \\
        \psi_i^k(x) &:= \dot{\psi}_i^{k-1}(x) + \alpha_i^r(\psi_i^{k-1}(x))
    \end{aligned}
\end{equation}
where $\alpha_i^1(\cdot),\alpha_i^1(\cdot), \dots, \alpha_i^k(\cdot)$ denote class-$\mathcal{K}$ functions of their argument. These functions provide definitions for the corresponding series of sets
\begin{equation} \label{eq:HO_sets}
    \begin{aligned}
        \mathcal{C}_i^1 &:= \{ x \in \mathbb{R}^N: \psi_i^0(x) \geq 0 \} \\
        \mathcal{C}_i^2 &:= \{ x \in \mathbb{R}^N: \psi_i^1(x) \geq 0 \} \\
        & \vdots \\
        \mathcal{C}_i^k &:= \{ x \in \mathbb{R}^N: \psi_i^{k-1}(x) \geq 0 \} \\
    \end{aligned}
\end{equation}
which yield the following definition.
\begin{definition}
    Let $\mathcal{C}_i^1, \mathcal{C}_i^2, \dots, \mathcal{C}_i^k$ be defined by \eqref{eq:HO_funcs} and \eqref{eq:HO_sets}. We have that $h_i$ is a \textit{node-level barrier function} (NBF) for node $i\in [n]$ if $h_i \in C^k$ and there exist differentiable class-$\mathcal{K}$ functions $\alpha_i^1,\alpha_i^2,\dots, \alpha_i^k$ such that $\psi_i^k(x) \geq 0$ for all $x \in \bigcap_{r=1}^k \mathcal{C}_i^r$. 
\end{definition}
\noindent
This definition leads naturally to the following lemma (which is a direct result of Theorem 4 in \cite{xiao2019control}).
\begin{lemma} \label{lem:NBF}
If $h_i$ is an NBF, then $\bigcap_{r=1}^k \mathcal{C}_i^r$ is forward invariant.
\end{lemma}

\section{Safe Formation Control Problem} \label{sec:safe_form_prob}

In this section, we define a general version of the safe formation control problem with respect to applying a safety filter to control actions that affect individual agent behavior governed by assumed formation dynamics.
For the sake of notational brevity, we use $x$, the full state of the network, and $(x_i, x_{\mathcal{N}_i})$, the concatenated states of agents in the neighborhood centered on agent $i \in [n]$, interchangeably moving forward.
Consider the first-order dynamics for a single agent~$i$
\begin{equation} \label{eq:dynamics}
    \dot{x}_i = f_i(x_i) + g_i(x_i)u_i
\end{equation}
where $u_i \in \mathcal{U}_i \subset \mathbb{R}^{M_i}$ is some form of affine acceleration controller for agent $i$.  
Let $u_i^f(x_i, x_{\mathcal{N}_i})$ be a distributed feedback control law that induces some formation behavior. We can treat these formation dynamics  as part of the natural dynamics of the system where $u_i^f(x_i, x_{\mathcal{N}_i})$ is modified by some safety filter control law as
\begin{equation*}
    \dot{x}_i = f_i(x_i) + g_i(x_i)(u_i^f(x_i, x_{\mathcal{N}_i})-u_i^s)
\end{equation*}
where $u_i^s$ is a modification to the formation control signal to ensure agent safety. We can then rewrite the dynamics in \eqref{eq:dynamics} as
\begin{equation} \label{eq:formation_dynamics}
    \dot{x}_i = \bar{f}_i(x_i, x_{\mathcal{N}_i}) + \bar{g}_i(x_i)u_i^s
\end{equation}
where
\begin{equation}
    \bar{f}_i(x_i, x_{\mathcal{N}_i}) = f_i(x_i) + g_i(x_i)u_i^f(x_i, x_{\mathcal{N}_i})
\end{equation}
and
\begin{equation}
    \bar{g}_i(x_i) = -g_i(x_i).
\end{equation}

We assume each agent has positional safety constraints with respect to a given obstacle $o \in \mathcal{O}_i(t)$, where $\mathcal{O}_i(t)$ is the set of identifiers for obstacles within the sensing range of agent $i$ at time $t$. Note that other agents within the sensing range of agent $i$ at time $t$ will also be included in $\mathcal{O}_i(t)$, which does not change the computation of the first-order safety condition. However, if agents in $\mathcal{O}_i(t)$ are also in $\mathcal{N}_i$, then the expression for the second-order safety condition for inter-agent collision avoidance with respect to the defined formation dynamics, which will be explained in further detail in Section~\ref{sec:safe_with_colab}, must also incorporate partial derivatives with respect to $x_j$ for $j \in \mathcal{N}_i$. For convenience, we drop the notation of time dependence on $\mathcal{O}_i$ moving forward. We define the set of viable safety filter control actions as
\begin{equation}
    \mathcal{U}_i^s(x) = \{ u_i^s \in \mathcal{U}_i: u_i^f(x) - u_i^s \in \mathcal{U}_i \}.
\end{equation}
In this paper, we assume safety conditions for each agent are defined with respect to the relative position of agents to obstacles. Therefore, since control is implemented through acceleration, we construct a higher-order barrier function for each agent $i$ with respect to a given obstacle $o$ as follows
\begin{equation}\label{eq:second_order_BF}
    \begin{aligned}
        \phi_{i,o}^0(x_i, x_o) &= h_i(x_i, x_o) \\
        \phi_{i,o}^1(x_i, x_o) &= \dot{\phi}_{i,o}^0(x_i, x_o) + \alpha_i^0(\phi_{i,o}^0(x_i, x_o))
    \end{aligned}
\end{equation}
where $x_o$ is the state of obstacle $o \in \mathcal{O}_i$. These functions then define the corresponding safety constraint sets
\begin{equation} \label{eq:HO_obst_sets}
    \begin{aligned}
        \mathcal{C}_{i,o}^1 &:= \{ (x_i, x_o) \in \mathbb{R}^{N_i} \times \mathbb{R}^{N_o}: \phi_{i,o}^0(x_i, x_o) \geq 0 \} \\
        \mathcal{C}_{i,o}^2 &:= \{ (x_i, x_o) \in \mathbb{R}^{N_i} \times \mathbb{R}^{N_o}: \phi_{i,o}^1(x_i, x_o) \geq 0 \}.
    \end{aligned}
\end{equation}

\noindent Given the definition of these constraint sets, we can define an \textit{agent-level control barrier function} and subsequent forward invariant properties as follows.

\begin{definition}
    We have $h_{i,o}(x_i, x_o)$ is an \textit{agent-level control barrier function} (aCBF) if for all $(x_i, x_o) \in \mathcal{C}_{i,o}^1 \cap \mathcal{C}_{i,o}^2$ and $t \in \mathcal{T}$ there exists a class-$\mathcal{K}$ function $\alpha_i^1$ and $u_i^s \in \mathcal{U}_i^s(x)$ such that
    \begin{equation}\label{eq:first_order_safety_cond}
        \dot{\phi}_{i,o}^1(x, x_o, u_i^f(x), u_i^s) + \alpha_i^1(\phi_{i,o}^1(x_i, x_o)) \geq 0.
    \end{equation}
\end{definition}
\noindent
We see that \eqref{eq:first_order_safety_cond} characterizes the first-order safety condition for agent $i$ with respect to obstacle $o$ since the acceleration control input appears in the second derivative of $h_{i,o}$, which is computed in $\dot{\phi}_{i,o}^1$. This barrier function definition naturally leads to the following result on agent-level safety.

\begin{lemma}\label{lem:aCBF}
    If $h_{i,o}(x_i, x_o)$ is an aCBF, then $\mathcal{C}_{i,o}^1 \cap \mathcal{C}_{i,o}^2$ is forward invariant for all $t \in \mathcal{T}$.
\end{lemma}
\begin{proof}
    If $h_{i,o}$ is an aCBF, then $\exists u_i^s \in \mathcal{U}_i^s$ such that
    \begin{equation*}
        \dot{\phi}_{i,o}^1(x, x_o, u_i^f(x), u_i^s) + \alpha_i^1(\phi_{i,o}^1(x_i, x_o)) \geq 0
    \end{equation*}
    for all $(x_i, x_o) \in \mathcal{C}_{i,o}^1 \cap \mathcal{C}_{i,o}^2$. Thus, as $\phi_{i,o}^1(x_i, x_o)$ approaches zero there will be some $u_i^s$ such that $\dot{\phi}_{i,o}^1(x, x_o, u_i^f(x), u_i^s) \geq 0$. Therefore, if $(x_i(t_0), x_o(t_0)) \in \mathcal{C}_{i,o}^2$ then $\mathcal{C}_{i,o}^2$ is forward invariant for all $t \in \mathcal{T}$. By Lemma~\ref{lem:NBF}, it naturally follows that is $\mathcal{C}_{i,o}^1 \cap \mathcal{C}_{i,o}^2$ forward invariant.
\end{proof}

With agent-level control barrier functions defined, we are now prepared to state our formal problem for this work, which is defined in our notation as follows: 
\begin{equation} \label{eq:problem_statement}
\small
    \begin{aligned}
        \min_{u_i^s \in \mathcal{U}_i^s(x)} \quad & \frac{1}{2}{\left\Vert u_i^f(x_i, x_{\mathcal{N}_i}) - u_i^s \right\Vert}_2^2 \\
        \text{s.t.} \quad & \dot{\phi}_{i,o}^1\left(x, x_o, u_i^f, u_i^s \right) + \alpha_i^1\left(\phi_{i,o}^1(x_i, x_o)\right) \geq 0 \\
        & \forall i \in [n], \; \forall o \in \mathcal{O}_i.
    \end{aligned}
\end{equation}

\noindent
In words, we aim to provide a control policy that minimally alters the prescribed distributed formation control signal such that the defined safety conditions for obstacle avoidance are satisfied for all agents in the formation. 

We present our solution to this problem in the following sections, where in Section~\ref{sec:safe_with_colab} we define a barrier function candidate based on the relative positions of agents to obstacles and leverage previous work in \cite{butler2023distributed} to define a second-order safety condition that includes the effect of neighbor's formation dynamics on agent safety. We then modify the \textit{collaborative safety algorithm} from \cite{butler2023distributed} in Section~\ref{sec:colab_algo} to account for communicating safety needs with multiple safety constraints and provide a method for computing the maximum agent safety capability in Section~\ref{sec:max_cap_comp}, culminating in a modified collaborative safety algorithm for formation control in Section~\ref{sec:colab_safety_thm}. We then demonstrate our collaborative safety algorithm in simulation on a distributed formation controller for two-dimensional agents in Section~\ref{sec:application}.

\section{Safe Formation Control with Collaboration} \label{sec:safe_with_colab}

We now present a method by which each agent can communicate safety needs to its neighboring agents to achieve collective safety in a distributed manner.
We define a relative position safety constraint for each agent with respect to a given obstacle as follows. Let $p_i$ and $p_o$ be the position of agent $i \in [n]$ and obstacle $o \in \mathcal{O}_i$, respectively. We define a position based safety constraint as 
\begin{equation} \label{eq:rel_dist_barrierfunc}
    h_{i,o}(x_i, x_o) = {\Vert p_i - p_o \Vert}^2_2 - r_{i,o}^2
\end{equation}
where $r_{i,o} \in \mathbb{R}$ is the minimum distance agent $i$ should maintain from obstacle $o$. Assuming control inputs on the acceleration of agent $i$, we use the second-order barrier functions candidate from \eqref{eq:second_order_BF} to define the first derivative safety condition
\begin{equation}
    \dot{\phi}_{i,o}^1(x, x_o, u_i^s) = \mathcal{L}_{\bar{f}_i}\phi_{i,o}^1(x, x_o) + \mathcal{L}_{\bar{g}_i}\phi_{i,o}^1(x_i, x_o) u_i^s.
\end{equation}
If we define the next high-order barrier function as
\begin{equation} \label{eq:phi2}
    \phi_{i,o}^2(x, x_o, u_i^s) = \dot{\phi}_{i,o}^1(x, x_o, u_i^f, u_i^s) + \alpha_i^1(\phi_{i,o}^1(x_i, x_o))
\end{equation}
and
\begin{equation} \label{eq:sec_order_safety_cond}
    \begin{aligned}
        \Phi_{i,o}(x, x_o, u_i^s, u_{\mathcal{N}_i}^s) = \dot{\phi}_{i,o}^2(x, x_o, u_i^s, u_{\mathcal{N}_i}^s)
        + \alpha_i^2(\phi_{i,o}^2(x, x_o, u_i^s)),
    \end{aligned}
\end{equation}
we begin to see neighbor dynamics and the subsequent effect of neighbor control actions in the higher-order derivative expressions. A more detailed discussion on the derivation of \eqref{eq:sec_order_safety_cond} may be found in \cite{butler2023distributed}; however, for our purposes, we separate \eqref{eq:sec_order_safety_cond} into terms that are affected by neighbor control and those that are not affected by neighbor control as follows
\begin{equation} \label{eq:sec_order_safety_cond_separated}
    \Phi_{i,o}(x, x_o, u_i^s, u_{\mathcal{N}_i}^s) = \sum_{j \in \mathcal{N}_i} a_{ij,o}(x, x_o)u_j^s + c_{i,o}(x, x_o, u_i^s)
\end{equation}
where
\begin{equation} \label{eq:neighbor_effects}
    a_{ij,o}(x, x_o) = \mathcal{L}_{\bar{g}_j}\mathcal{L}_{\bar{f}_i} \phi_{i,o}^1(x, x_o)
\end{equation}
is the effect that modified control actions $u_j^s$ taken by agent $j \in \mathcal{N}_i$ have on the formation dynamics and the subsequent safety condition of agent $i$ with respect to obstacle $o \in \mathcal{O}_i$ and $c_{i,o}(x, x_o, u_i^s)$ collects all other terms including those that are affected by its own control actions $u_i^s$. Note that if neighbors in $\mathcal{N}_i$ also implement control through acceleration inputs then it is possible for $a_{ij,o}=\mathbf{0}^{M_i}$ since control inputs for neighbors do not appear until the next order barrier function. In this case, we can circumvent the need to compute unnecessary derivatives by having agents communicate safety needs in terms of velocity constraints, which may be used to approximate acceleration constraints locally for each agent. We will give an example of how this approximation may be done in practice in Section~\ref{sec:application}.

To compute $c_{i,o}$ more explicitly, we make the following assumption,
\begin{assumption} \label{assume:class_K_scalar}
    Let $\alpha_i^1(z) := \alpha_i^1 z$ and $\alpha_i^2(z) := \alpha_i^2 z$ where $z \in \mathbb{R}^{N_i}$ and $\alpha_i^1, \alpha_i^2 \in \mathbb{R}_{> 0}$ and define $\beta_i = \alpha_i^1 +  \alpha_i^2$. 
\end{assumption}

\noindent
This assumption yields the full expression of $c_{i,o}$ as
\begin{equation} \label{eq:safety_capability}
\small
    \begin{aligned}
        c_{i,o}(x, x_o, u_i^s) &= \sum_{j \in \mathcal{N}_i} \mathcal{L}_{\bar{f}_j}\mathcal{L}_{\bar{f}_i} \phi_{i,o}^1(x, x_o) + \mathcal{L}_{\bar{f}_i}^2 \phi_{i,o}^1(x, x_o)
        + \alpha_i^1 \alpha_i^2 \phi_{i,o}^1(x_i, x_o) + \beta_i \mathcal{L}_{\bar{f}_i}\phi_{i,o}^1(x_i, x_o) + \mathcal{L}_{\bar{g}_i}\phi_{i,o}^1(x_i, x_o)\dot{u}_i^s \\
        & \quad + u_i^{s\top}\mathcal{L}_{\bar{g}_i}^2 \phi_{i,o}^1(x, x_o) u_i^s + \beta_i \mathcal{L}_{\bar{g}_i} \phi_{i,o}^1(x_i, x_o) u_i^s
        + \left[ \mathcal{L}_{\bar{f}_i} \mathcal{L}_{\bar{g}_i} \phi_{i,o}^1(x_i, x_o)^\top +\mathcal{L}_{\bar{g}_i} \mathcal{L}_{\bar{f}_i} \phi_{i,o}^1(x_i, x_o) \right]u_i^s.
    \end{aligned}
\end{equation}

\noindent
We may interpret \eqref{eq:safety_capability} as the total safety capability of agent~$i$ with respect to avoidance of obstacle $o \in \mathcal{O}_i$, where if $c_{i,o}(x, x_o, u_i^s) \geq 0$, then agent $i$ is capable of remaining safe given $u_i^s$ (assuming no negative action effects of neighbors). Conversely, if $c_{i,o}(x, x_o, u_i^s) < 0$, then agent $i$ is incapable of remaining safe given $u_i^s$ and will require assistance from its neighbor's actions.
Note that in the case where $x_o = x_j$ for some $j \in \mathcal{N}_i$, the computation of \eqref{eq:sec_order_safety_cond} must also incorporate the dependence of $x_j$ in $\Vert p_i - p_j \Vert$ when evaluating the Lie derivatives in both \eqref{eq:neighbor_effects} and \eqref{eq:safety_capability}.
Given our definition of a subsequent higher-order barrier function in \eqref{eq:phi2}, we define another safety constraint set as
\begin{equation} \label{eq:HO_neighbor_set}
    \begin{aligned}
        \mathcal{C}_{i,o}^3 := & \big\{ (x_i, x_o) \in \mathbb{R}^{N_i} \times \mathbb{R}^{N_o}: \exists u_i^s \in \mathcal{U}_i^s \text{ s.t. } 
        \dot{\phi}_{i,o}^1 (x, x_o, u_i^f, u_i^s) + \alpha_i^1\left(\phi_{i,o}^1(x_i, x_o)\right) \geq 0 \big\}
    \end{aligned}
\end{equation}
which collects all states where agent $i$ is capable of maintaining its first-order safety condition under the influence of its induced formation dynamics. Given these definitions, we are prepared to define a collaborative control barrier function as follows.
\begin{definition}
    Let $\mathcal{C}_{i,o}^1$, $\mathcal{C}_{i,o}^2$, and  $\mathcal{C}_{i,o}^3$ be defined by \eqref{eq:HO_obst_sets} and \eqref{eq:HO_neighbor_set}. We have that $h_{i,o}$ is a \textit{collaborative control barrier function} (CCBF) for node $i \in [n]$ if $h_{i,o} \in C^3$ and $\forall (x_i,x_o) \in \mathcal{C}_{i,o}^1 \cap \mathcal{C}_{i,o}^2 \cap \mathcal{C}_{i,o}^3$ and $\forall t \in \mathcal{T}$ there exists $(u_i^s, u_{\mathcal{N}_i}^s) \in \mathcal{U}_i^s \times \mathcal{U}_{\mathcal{N}_i}^s$ such that 
    \begin{equation} \label{eq:CCBF_cond}
        \Phi_{i,o}(x, x_o, u_i^s, u_{\mathcal{N}_i}^s) \geq 0, \; \forall o \in \mathcal{O}_i.
    \end{equation}
\end{definition}

\begin{lemma} \label{lem:CCBF}
    Given a distributed multi-agent system defined by \eqref{eq:formation_dynamics} and constraint sets defined by \eqref{eq:HO_obst_sets} and \eqref{eq:HO_neighbor_set}, $\bigcap_{o \in \mathcal{O}_i}\mathcal{C}_{i,o}^1 \cap \mathcal{C}_{i,o}^2 \cap \mathcal{C}_{i,o}^3$ is forward invariant $\forall t \in \mathcal{T}$ if $h_{i,o}$ is a CCBF for all $o \in \mathcal{O}_i$. 
\end{lemma}
\begin{proof}
    The results of this lemma are a direct extension of Theorem 2 in \cite{butler2023distributed}, where if $h_{i,o}$ is a CCBF for a given obstacle $o \in \mathcal{O}_i$ then $\exists (u_i^s, u_{\mathcal{N}_i}^s) \in \mathcal{U}_i^s \times \mathcal{U}_{\mathcal{N}_i}^s$ such that \eqref{eq:CCBF_cond} holds.
    Since $u_i^s$ appears in both $\Phi_{i,o}(x, x_o, u_i^s, u_{\mathcal{N}_i}^s)$ and $\phi_{i,o}^2(x, x_o, u_i^s)$, we must show that if $(x_i,x_o) \in \mathcal{C}_{i,o}^1 \cap \mathcal{C}_{i,o}^2 \cap \mathcal{C}_{i,o}^3$ and $\Phi_{i,o}(x, x_o, u_i^s, u_{\mathcal{N}_i}^s) \geq 0$ for some $u_i^s \in \mathcal{U}_i^s$, then $\phi_{i,o}^2(x, x_o, u_i^s) \geq 0$ also. If \eqref{eq:CCBF_cond} holds for all $(x_i,x_o) \in \mathcal{C}_{i,o}^1 \cap \mathcal{C}_{i,o}^2 \cap \mathcal{C}_{i,o}^3$, then for all $x_i, x_{\mathcal{N}_i}, x_o$ and $u_i^s \in \mathcal{U}_i^s$ where $\phi_{i,o}^2(x, x_o, u_i^s) = 0$, there exists $u_{\mathcal{N}_i}^s \in \mathcal{U}_{\mathcal{N}_i}^s$ such that $\dot{\phi}_{i,o}^2(x, x_o, u_i^s, u_{\mathcal{N}_i}^s) \geq 0$. Thus, we have that $\phi_{i,o}^2(x, x_o, u_i^s) \geq 0, \forall (x_i,x_o) \in \mathcal{C}_{i,o}^1 \cap \mathcal{C}_{i,o}^2 \cap \mathcal{C}_{i,o}^3$, which implies $\phi_i^1(x_i, x_o) \geq 0, \forall (x_i,x_o) \in \mathcal{C}_{i,o}^1 \cap \mathcal{C}_{i,o}^2 \cap \mathcal{C}_{i,o}^3$. Therefore, we have that $\mathcal{C}_{i,o}^1 \cap \mathcal{C}_{i,o}^2 \cap \mathcal{C}_{i,o}^3$ is forward invariant. Further, since these same arguments hold $\forall o \in \mathcal{O}_i$, 
    it directly follows that $\bigcap_{o \in \mathcal{O}_i}\mathcal{C}_{i,o}^1 \cap \mathcal{C}_{i,o}^2 \cap \mathcal{C}_{i,o}^3$ is forward invariant.
\end{proof}

With set invariance defined with respect to neighbor influence, we can leverage these properties to construct an algorithm to implement collaborative safety through rounds of communication between neighbors.

\subsection{Multi-Agent Collaboration Through Communication} \label{sec:colab_algo}
In this section, we introduce the \textit{collaborative safety algorithm}, modified from our previous work in \cite{butler2023distributed}. The major additional contribution to the algorithm in this work is the additional handling of multiple safety constraints from each agent, which requires a new definition of maximum safety capability with respect to multiple safety conditions. For the formation control problem scenario, we make the following assumptions.
\begin{assumption} \label{assume:u_i_piecewise_constant}
    Let $u_i^s(t)$ be piecewise constant $\forall t \in \mathcal{T}$.
\end{assumption}
\noindent
This assumption includes zero-hold controllers that implement control decisions in a bang-bang fashion, allowing us to set $\dot{u}_i=0$ in the analysis.
\begin{assumption} \label{assume:Lgi^2_is_0}
    Let $\mathcal{L}_{\bar{g}_i}^2 \phi_{i,o}^1(x_i, x_o) = \mathbf{0}^{M_i \times M_i}, \forall o \in \mathcal{O}_i$.
\end{assumption}
\noindent
In words, we assume that the control exerted by agent $i$ does not have a dynamic relationship with its ability to exert control (e.g., the robot's movement is implemented identically no matter its position in a defined coordinate system). Since each agent may be actively avoiding multiple obstacles, we may compute the vector describing the second-order safety condition with respect to each obstacle under Assumptions~\ref{assume:class_K_scalar}-\ref{assume:Lgi^2_is_0} as follows
\begin{equation} \label{eq:capability_vector_expanded}
    \begin{aligned}
        \Phi_i &=
        \underbrace{
        \begin{bmatrix}
            a_{ij_1,o_1}(x, x_{o_1}) & \cdots & a_{ij_{|\mathcal{N}_i|},o_1}(x, x_{o_1}) \\
            \vdots & \ddots & \vdots \\
            a_{ij_1,o_{K}}(x, x_{o_K}) & \cdots & a_{ij_{|\mathcal{N}_i|},o_{K}}(x, x_{o_K})
        \end{bmatrix}
        }_{A_i}
        \begin{bmatrix}
            u_{j_1}^s \\ \vdots \\ u_{j_{|\mathcal{N}_i|}}^s
        \end{bmatrix}
        +
        \underbrace{
        \begin{bmatrix}
            \mathcal{L}_{\bar{f}_i} \mathcal{L}_{\bar{g}_i} \phi_{io_{1}}^{1\top} + \mathcal{L}_{\bar{g}_i} \mathcal{L}_{\bar{f}_i} \phi_{io_{1}}^1 + \beta_i \mathcal{L}_{\bar{g}_i} \phi_{io_{1}}^1 \\
            \vdots \\
            \mathcal{L}_{\bar{f}_i} \mathcal{L}_{\bar{g}_i} \phi_{io_{K}}^{1\top} + \mathcal{L}_{\bar{g}_i} \mathcal{L}_{\bar{f}_i} \phi_{io_{K}}^1 + \beta_i \mathcal{L}_{\bar{g}_i} \phi_{io_{K}}^1
        \end{bmatrix}
        }_{B_i}
        u_i^s 
        \\
        & \quad +
        \underbrace{
        \begin{bmatrix}
            \sum_{j \in \mathcal{N}_i} \mathcal{L}_{\bar{f}_j}\mathcal{L}_{\bar{f}_i} \phi_{io_{1}}^1 + \mathcal{L}_{\bar{f}_i}^2 \phi_{io_{1}}^1 + \alpha_i^1 \alpha_i^2 \phi_{io_{1}}^1 + \beta_i \mathcal{L}_{\bar{f}_i}\phi_{io_{1}}^1 \\
            \vdots \\
            \sum_{j \in \mathcal{N}_i} \mathcal{L}_{\bar{f}_j}\mathcal{L}_{\bar{f}_i} \phi_{io_{K}}^1 + \mathcal{L}_{\bar{f}_i}^2 \phi_{io_{K}}^1 + \alpha_i^1 \alpha_i^2 \phi_{io_{K}}^1 + \beta_i \mathcal{L}_{\bar{f}_i}\phi_{io_{K}}^1
        \end{bmatrix}
        }_{q_i}
    \end{aligned}
\end{equation}
where $K = |\mathcal{O}_i(t)|$ is the number of obstacles within the sensing range of agent $i$ at time $t$. Note our early remark that other agents in the formation within the sensing range of agent $i$ will also be included in this vector to account for inter-agent collision avoidance. Further, note that the length of this vector is time-varying according to $|\mathcal{O}_i(t)|$. We can express \eqref{eq:capability_vector_expanded} more compactly as
\begin{equation} \label{eq:multi-obs-compact}
    \Phi_i(x, x_o, \forall o \in \mathcal{O}_i) = A_i u_{\mathcal{N}_i}^s + B_i u_i^s + q_i
\end{equation}
where $A_i \in \mathbb{R}^{K \times M_{\mathcal{N}_i}}$, with $M_{\mathcal{N}_i} = \sum_{j \in \mathcal{N}_i} M_j$, $B_i \in \mathbb{R}^{K \times M_i}$, and $q_i \in \mathbb{R}^{K}$.
Under \eqref{eq:capability_vector_expanded}, we have the following result on its relationship to the problem stated in \eqref{eq:problem_statement}. 
\begin{lemma} \label{lem:equivalent_solution}
    Under Assumptions~\ref{assume:class_K_scalar}-\ref{assume:Lgi^2_is_0}, any set of agent control inputs  $u_i^s \in \mathcal{U}_i^s, \forall i \in [n]$ that satisfies
    \begin{equation} \label{eq:lem_compact_condition}
       \Phi_i(x, x_o, u_i^s, u_{\mathcal{N}_i}^s, \forall o \in \mathcal{O}_i) \geq \mathbf{0}; \forall i \in [n]
    \end{equation}
    from \eqref{eq:multi-obs-compact} are also a solution to 
    \begin{equation}\label{eq:lem_prob_statement}
        \dot{\phi}_{i,o}^1 (x, x_o, u_i^f, u_i^s) + \alpha_i^1\left(\phi_{i,o}^1(x_i, x_o)\right) \geq 0; \forall i \in [n], \forall o \in \mathcal{O}_i
    \end{equation}
    from \eqref{eq:problem_statement}.
\end{lemma}
\begin{proof}
    By the proof of Lemma~\ref{lem:CCBF}, if $\Phi_{i,o}(x, x_o, u_i^s, u_{\mathcal{N}_i}^s) \geq 0$ for some $u_i^s \in \mathcal{U}_i^s$, then $\phi_{i,o}^2(x, x_o, u_i^s) = \dot{\phi}_{i,o}^1 (x, x_o, u_i^f, u_i^s) + \alpha_i^1\left(\phi_{i,o}^1(x_i, x_o)\right) \geq 0$ also. Thus, since \eqref{eq:lem_compact_condition} implies that $\Phi_{i,o}(x, x_o, u_i^s, u_{\mathcal{N}_i}^s) \geq 0, \forall o \in \mathcal{O}_i$, then by Assumptions~\ref{assume:class_K_scalar}-\ref{assume:Lgi^2_is_0}, which simplify the expression of \eqref{eq:sec_order_safety_cond_separated} by selecting scalar class-$\mathcal{K}$ functions by Assumption~\ref{assume:class_K_scalar}, setting $\mathcal{L}_{\bar{g}_i}\phi_{i,o}^1(x_i, x_o)\dot{u}_i^s = 0$ by Assumption~\ref{assume:u_i_piecewise_constant}, and setting $u_i^{s\top}\mathcal{L}_{\bar{g}_i}^2 \phi_{i,o}^1(x, x_o) u_i^s = 0, \forall o\in \mathcal{O}_i$ by Assumption~\ref{assume:Lgi^2_is_0}, the set of agent control inputs that satisfy \eqref{eq:lem_compact_condition} must also satisfy \eqref{eq:lem_prob_statement}.
\end{proof}

We now describe the collaborative safety algorithm and how it may be used to communicate safety needs to neighboring agents in the formation control problem. See \cite{butler2023distributed} for a more detailed discussion on the construction of the collaborative safety algorithm with respect to a single safety condition for each agent. The central idea of this algorithm involves rounds of communication between agents, where each round of communication between agents, centered on an agent $i\in [n]$, involves the following steps:
\begin{enumerate}
    \item Receive (send) requests from (to) neighbors in $\mathcal{N}_i$
    \item Process requests and determine needed compromises
    \item Send (receive) adjustments to (from) neighboring nodes in $\mathcal{N}_i$.
\end{enumerate}
The end result of this algorithm will be some set of constrained allowable filtered actions for each agent $\overline{\underline{\mathcal{U}}}_i^s \subseteq \mathcal{U}_i^s$, where any safe action selected from this set will also be safe for all neighbors in $\mathcal{N}_i$.
In order to determine what requests should be made of neighbors, each agent must compute its maximum safety capability with respect to the second-order safety condition as defined by \eqref{eq:sec_order_safety_cond}. However, since the safety capability of agent $i$ with respect to multiple obstacles is represented as a vector, rather than a scalar value for a single condition (\cite{butler2023distributed}), we must carefully define the maximum safety capability for agents in the context of formation control with multiple obstacles. Therefore, in the following section, we define a method for determining the vector of maximum capability for agent $i$ with respect to multiple obstacles and how this information may be used to communicate its safety needs to neighbors.

\subsection{Maximum Capability Given Multiple Obstacles} \label{sec:max_cap_comp}
To define the maximum capability of an agent $i$ with respect to multiple obstacles, we begin by making the following assumption. 
\begin{assumption} \label{assume:nonempty_convex_contraints}
Let $\mathcal{U}_i^s$ be a non-empty convex set which is defined by $\mathcal{U}_i^s = \{u_i^s \in \mathbb{R}^{M_i}: G_iu_i^s - l_i \leq \mathbf{0}\}$.
\end{assumption}
\noindent
In order to determine the ``safest" action agent $i$ may take given multiple obstacles, we want to choose the action $u_i^s$ that maximizes the minimum entry of the vector $B_i u_i^s$ from \eqref{eq:multi-obs-compact}, which is defined by the following max-min optimization problem: 
\begin{equation}\label{eq:opt_capcity_input_action}
    \max_{u_i^s \in \mathcal{U}_i^s} \min_{1\leq k \leq |\mathcal{O}_i|} [B_i u_i^s]_{k}.
\end{equation}
This problem characterizes the optimal control strategy $u^*_i$ that attempts to satisfy the safety constraint \eqref{eq:sec_order_safety_cond_separated} imposed on agent $i$ for each obstacle $o \in \mathcal{O}_i$ that is at most risk of being violated (or being violated the worst).
We can reduce \eqref{eq:opt_capcity_input_action} to a linear programming problem:
\begin{align}\label{eq:opt_capcity_input_action_LP}
    \min_{\xi_i}&~d^\top \xi_i \nonumber\\
    \text{s.t.}& \begin{bmatrix}
        \mathbf{0} & G_i\\
        \mathbf{1} & -B_i
    \end{bmatrix}\xi_i - \begin{bmatrix}
        l_i\\
        \mathbf{0}
    \end{bmatrix}\leq \mathbf{0},
\end{align}
where $d^\top = \begin{bmatrix}
    -1 & \mathbf{0}_{M_i}^\top
\end{bmatrix}$, $\xi_i^\top = \begin{bmatrix}
    \gamma_i & u_i^\top
\end{bmatrix}$, and $\gamma_i \in \mathbb{R}$ is a scalar that captures the performance of the optimal strategy~$u^*_i$.
The next proposition formally characterizes the equivalency of Problem~\eqref{eq:opt_capcity_input_action} and Problem~\eqref{eq:opt_capcity_input_action_LP}.
\begin{proposition}\label{pp:equiv_opt_prob}
    Given Assumptions~\ref{assume:class_K_scalar}-\ref{assume:nonempty_convex_contraints}, the optimal solution of \eqref{eq:opt_capcity_input_action}:
    \begin{align*}
        u_i^* &= \arg\max_{u_i^s \in \mathcal{U}_i^s} \min_{1\leq k \leq |\mathcal{O}_i|} [B_i u_i^s]_{k}\\
        \gamma_i^* &= \max_{u_i^s \in \mathcal{U}_i^s} \min_{1\leq k\leq |\mathcal{O}_i|} [B_i u_i^s]_{k}
    \end{align*}
    exists if and only if there exists an optimal solution in \eqref{eq:opt_capcity_input_action_LP}, $\xi_i^{(*)} = \begin{bmatrix}
        \gamma_i^{(*)}& {u_i^{(*)}}^\top
    \end{bmatrix}^\top$, 
    and 
    $\gamma_i^{(*)} = \gamma_i^*$, $u_i^{(*)} = u_i^*$.
\end{proposition}
\begin{proof}
    We first notice that the following two optimization problems are equivalent:
    \begin{align}
        &\begin{cases}\label{eq:pp:step1}
            \max_{u_i^s}\min_{1\leq k \leq |\mathcal{O}_i|}~&[B_i u_i^s]_{k}\\
            \text{s.t.}~& G_iu_i^s - l_i \leq \mathbf{0},\\
        \end{cases}\\
        &\begin{cases}\label{eq:pp:step2}
            \max_{u_i^s, \gamma_i}~&\gamma_i\\
            \text{s.t.}~&G_iu_i^s - l_i \leq \mathbf{0}\\
                       & \gamma_i \leq [B_i u_i^s]_{k}~\qquad\forall 1\leq k \leq |\mathcal{O}_i|,
        \end{cases}
    \end{align}
    by substituting $\min_{1\leq k \leq |\mathcal{O}_i|}~[A_i(x_i)u_i^s]_{k}$ with an achievable lower bound $\gamma_i$ on each $[A_i(x_i)u_i^s]_{k}$, that is, $\gamma_i \leq [A_i(x_i)u_i^s]_{k}~\forall 1\leq k \leq |\mathcal{O}_i|$.
    Furthermore, we can show that
    \eqref{eq:opt_capcity_input_action} is equivalent to \eqref{eq:pp:step1} by rewriting $u_i^s \in \mathcal{U}_i^s$ explicitly as an optimization constraint, and similarly, we can show that
    \eqref{eq:opt_capcity_input_action_LP} is equivalent to \eqref{eq:pp:step2} by setting $\xi_i = \begin{bmatrix}
        \gamma_i & u_i^{s\top}
    \end{bmatrix}^\top$, $d^\top = \begin{bmatrix}
        -1 & 0_{M_i}^\top
    \end{bmatrix}$, and realizing that $\argmax \gamma_i = \argmin -\gamma_i$.
    The transitivity of equivalence relations concludes the proof. 
\end{proof}

Thus, we have a method for computing a vector that represents the maximum capability of agent $i$ with respect to multiple obstacles $\mathcal{O}_i$. If $\gamma_i$ is negative, then agent $i$ will make a request to its neighboring agents that will limit their control actions $\overline{\underline{\mathcal{U}}}_j^s$ to those that will satisfy $[\Phi_i]_k \geq 0, \forall k \in \mathcal{O}_i$, assuming agent $i$ takes the action $u_i^*$. In the following section, we describe how our \textit{modified collaborative safety algorithm} incorporates this capability vector at a high level. 

\subsection{Collective Safety Through Distributed Collaboration} \label{sec:colab_safety_thm}
Given our addition to the collaborative safety algorithm from \cite{butler2023distributed} to incorporate multiple safety constraints, the computation steps and convergence properties of our algorithm remain largely unchanged in Algorithm~\ref{alg:mod_colab_safety} due to the fact that communication of multiple safety constraints from one neighbor is equivalent to multiple neighbors communicating a single constraint in the computation of control restrictions.

\begin{algorithm}
    \caption{Modified Collaborative Safety}\label{alg:mod_colab_safety}
    \begin{algorithmic}[1]
        \State $\bar{c}_{ij} \gets \mathbf{0}, \forall j \in \mathcal{N}_i$
        \State $\overline{\underline{\mathcal{U}}}_i^s \gets \mathcal{U}_i^s$
        \Repeat
            \State $\bar{c}_i \gets$ Compute maximum capability by solving \eqref{eq:opt_capcity_input_action_LP}
            \State $\delta_i \gets \bar{c}_i - \sum_{j \in \mathcal{N}_i}\bar{c}_{ij}$
            \State $\bar{c}_{ij}, \; \overline{\underline{\mathcal{U}}}_i^s \gets$ Perform SPRU (\cite{butler2023distributed})
        \Until{$\bar{c}_{ij}$ remains constant \textbf{and} $[\delta_i]_{k \in \left[|\mathcal{O}_i| \right]} \geq 0$}
    \end{algorithmic}
\end{algorithm}
\noindent
We denote SPRU as an abbreviation of (S)end/receive requests, (P)rocess requests, (R)ecieve/send adjustments, (U)pdate constraints w.r.t. adjustments as detailed in \cite{butler2023distributed}. We yield the following result on the collective safety of a formation under the modified collaborative safety algorithm. 

\begin{theorem} \label{thm:colab_safety_alg}
    Let Assumptions~\ref{assume:class_K_scalar}-\ref{assume:nonempty_convex_contraints} hold for all $i \in [n]$. If Algorithm~\ref{alg:mod_colab_safety} is convergent and $\overline{\underline{\mathcal{U}}}_i^s(x(t)) \neq \emptyset, \forall i \in [n], \forall t \in \mathcal{T}$, then \eqref{eq:problem_statement} yields $\bigcap_{o \in \mathcal{O}_i}\mathcal{C}_{i,o}^1 \cap \mathcal{C}_{i,o}^2$ forward invariant during $t \in \mathcal{T}$ for all $i \in [n]$.
\end{theorem}
\begin{proof}
If Algorithm~\ref{alg:mod_colab_safety} is convergent for all $t \in \mathcal{T}$ and there exists a non-empty $\overline{\underline{\mathcal{U}}}_i^s(x(t))$ for all $i \in [n]$, then by Lemmas~\ref{lem:CCBF} and \ref{lem:equivalent_solution} we have any action taken by any agent from these constrained control sets must render $\bigcap_{o \in \mathcal{O}_i}\mathcal{C}_{i,o}^1 \cap \mathcal{C}_{i,o}^2 \cap \mathcal{C}_{i,o}^3$ forward invariant for all $i \in [n]$. Thus by applying control constraints $\overline{\underline{\mathcal{U}}}_i^s(x(t))$ to \eqref{eq:problem_statement} for each agent, we have by Lemma~\ref{lem:aCBF} that $\bigcap_{o \in \mathcal{O}_i}\mathcal{C}_{i,o}^1 \cap \mathcal{C}_{i,o}^2$ is also forward invariant during $t \in \mathcal{T}$ for all $i \in [n]$.
\end{proof}

In words, we have that if Algorithm~\ref{alg:mod_colab_safety} always terminates with a feasible set of safe actions for all agents, then Theorem~\ref{thm:colab_safety_alg} guarantees that using \eqref{eq:problem_statement} to choose safe actions for individual agents renders all agents safe for all time. Note that \eqref{eq:problem_statement} filters the agent's actions according to their individual safety constraints, where the set of allowable actions $\overline{\underline{\mathcal{U}}}_i^s$ is given by Algorithm~\ref{alg:mod_colab_safety}.
It should also be noted that it is not guaranteed for the modified collaborative safety algorithm to converge in the case where there are conflicting requests (either between neighbors or between multiple safety conditions) which may be possible in obstacle-dense environments. Thus, providing conditions under which the collaborative safety algorithm remains convergent under conflicting requests in an important direction for future work.

\section{Application Example} \label{sec:application}
We now illustrate the application of our collaborative safety algorithm to the safe cooperative formation control of a simplified two-dimensional agent system and simulate a multi-obstacle avoidance scenario.
\subsection{Virtual Mass-Spring Formation Model}
Consider a two-dimensional multi-agent system with distributed formation control dynamics defined by a virtual mass-spring model, with $x_i = [p_i^{\vec{x}}, p_i^{\vec{y}}, v_i^{\vec{x}}, v_i^{\vec{y}}]^\top$
\begin{equation} \label{eq:formation_dynamics_mass_spring}
    \dot{x}_i =
    \begin{bmatrix}
        v_i^{\vec{x}} \\ v_i^{\vec{y}} \\ 0 \\ 0
    \end{bmatrix}
    +
    \begin{bmatrix}
        0 & 0 \\
        0 & 0 \\
        1 & 0 \\
        0 & 1 
    \end{bmatrix}
    \left( u_i^f(x) - u_i^s \right)
\end{equation}
where
\begin{equation} \label{eq:formation_controller_mass_spring}
   u_i^f(x) =
   \begin{bmatrix}
    u_i^{f_{\vec{x}}} \\
    u_i^{f_{\vec{y}}}
   \end{bmatrix}
   =
   \frac{1}{m_i}
   \begin{bmatrix}
    \sum_{j \in \mathcal{N}_i} k_{ij} s_{ij} \sin \theta_{ij} - b_{ij}v_i^{\vec{x}} \\
    \sum_{j \in \mathcal{N}_i} k_{ij} s_{ij} \cos \theta_{ij} - b_{ij}v_i^{\vec{y}} \\
   \end{bmatrix}
\end{equation}
describes the desired formation behavior of the system, where agents behave as if coupled by mass-less springs with $k_{ij}$ and $b_{ij}$ being the spring and dampening constants for the virtual spring from agent $j$ to agent $i$, respectively, and
\begin{equation*}
    s_{ij} = L_{ij} - R_{ij}
\end{equation*}
denoting the stretch length of a given spring connection with resting length $R_{ij}$ and
\begin{equation*}
    L_{ij} = \Vert p_i - p_j\Vert_2
\end{equation*}
being the current length of the spring. We compute the $\vec{x}$ and $\vec{y}$ components of the stretched spring as
\begin{equation*}
    \sin \theta_{ij} = \frac{p_i^{\vec{x}} - p_j^{\vec{x}}}{L_{ij}}, \;\; \cos \theta_{ij} = \frac{p_i^{\vec{y}} - p_j^{\vec{y}}}{L_{ij}}.
\end{equation*}

\noindent
Thus, our induced coupling model then becomes
\begin{equation}
    \bar{f}_i(x) =
    \begin{bmatrix}
        v_i^{\vec{x}} \\ v_i^{\vec{y}} \\ u_i^{f_{\vec{x}}} \\ u_i^{f_{\vec{y}}}
    \end{bmatrix}
    , \;\;
    \bar{g}_i = 
    \begin{bmatrix}
        0 & 0 \\
        0 & 0 \\
        -1 & 0 \\
        0 & -1 
    \end{bmatrix}
\end{equation}
with the first-order safety condition for a given obstacle using the barrier function candidate \eqref{eq:rel_dist_barrierfunc} computed as
\begin{equation}
    \begin{aligned}
        \phi_{i,o}^1(x_i, x_o) &= 2\left[v_i^{\vec{x}}(p_i^{\vec{x}} - p_o^{\vec{x}}) + v_i^{\vec{y}}(p_i^{\vec{y}} - p_o^{\vec{y}}) \right] 
         + \alpha_i^0(h_{i,o}(x_i, x_o))
    \end{aligned}
\end{equation}
which yields the Lie derivatives of the safety condition with respect to the formation dynamics as
\begin{equation}
    \begin{aligned}
        \mathcal{L}_{\bar{f}_i}\phi_{i,o}^1(x, x_o) &= 2 v_i^{\vec{x}}(v_i^{\vec{x}} + \alpha_i^0(p_i^{\vec{x}}-p_o^{\vec{x}}))
        + v_i^{\vec{y}}(v_i^{\vec{y}} + \alpha_i^0(p_i^{\vec{y}}-p_o^{\vec{y}}))
        + u_i^{f_{\vec{x}}}(p_i^{\vec{x}} - p_o^{\vec{x}}) + u_i^{f_{\vec{y}}}(p_i^{\vec{y}} - p_o^{\vec{y}}) 
    \end{aligned}
\end{equation}
and
\begin{equation}
    \mathcal{L}_{\bar{g}_i}\phi_{i,o}^1(x_i, x_o) = 2 
    \begin{bmatrix}
        p_i^{\vec{x}} - p_o^{\vec{x}} & p_i^{\vec{y}} - p_o^{\vec{y}}
    \end{bmatrix}.
\end{equation}

It should be noted that given this mass-spring network formation control law, when computing the effect of control by agent $j$ on the safety conditions of agent $i$
yields
\begin{equation}
    \mathcal{L}_{\bar{g}_j} \mathcal{L}_{\bar{f}_i} \phi_{i,o}^1(x_i, x_o) = \mathbf{0}^{M_j}
\end{equation}
since the control input of agent $j$ does not appear until the next derivative of $\Phi_i$. In order to avoid unnecessary computations of additional partial derivatives, each agent computes the effect of neighboring control as if neighbors directly control their velocities, i.e.,
\begin{equation}
    \bar{g}_j = 
    \begin{bmatrix}
        -1 & 0 \\
        0 & -1 \\
        0 & 0 \\
        0 & 0 
    \end{bmatrix}, \forall j \in \mathcal{N}_i.
\end{equation}
This assumption is non-physical since it would require infinite acceleration for neighbors to achieve such a discontinuous instantaneous jump in velocity. However, if we assume a finite time interval $\tau >0$ during which our acceleration controller might achieve such a change in velocity, we can approximate the necessary acceleration constraints during that time. In other words, since these terms are used to communicate action limitations on neighbors, we may approximate acceleration limits over a given time interval by simply dividing the velocity constraints by the appropriate time window length. Therefore, if the velocity constraints communicated are
\begin{equation}
    \mathcal{U}^v = \{ u \in \mathbb{R}^M: Gu +l \leq \mathbf{0} \},
\end{equation}
then we may compute acceleration constraints for a given time interval $\tau > 0$ as
\begin{equation}
    \mathcal{U}^a = \left\{ u \in \mathbb{R}^M: \frac{1}{\tau} Gu +l \leq \mathbf{0} \right\}.
\end{equation}
In real-world applications, this reliance on a known time interval may cause challenges when accounting for communication delays and inconsistent processing and actuation time intervals.

\subsection{Simulations}

We construct an example of multi-obstacle avoidance for a fully connected 3-agent formation where the parameters of the virtual mass-spring system are $m_i = 0.5$, $r_{i,o} = 1$, $K_{ij}=3$, $R_{ij} = 3$, and $b_{ij} = 1$ for all $i \in [n]$, $j \in \mathcal{N}_j$, and $o \in \mathcal{O}$. Further, we set control magnitude limits for each agent $i$ as $\mathcal{U}_i = \{ u_i \in \mathbb{R}^2: \Vert u_i \Vert_{\infty} \leq 15 \}$. We then apply a constant control signal to a single agent which leads the formation through an obstacle field, where the initial and final positions of each agent and their respective trajectories through the obstacle field are shown in Figure~\ref{fig:obs_field_trajectory}. Each agent uses the modified collaborative safety algorithm described in Algorithm~\ref{alg:mod_colab_safety} to communicate its safety needs and accommodate safety requests to and from neighbors, respectively. Each agent then implements a first-order safety filter on their control actions as described by \eqref{eq:problem_statement} while incorporating the control constraints $\overline{\underline{\mathcal{U}}}_i^s$ computed using Algorithm~\ref{alg:mod_colab_safety}. We plot the safety filter control signal including the constant leader control signal for agent $0$ in Figure~\ref{fig:control_safety_filter}, which shows the safety filter control signal $u_i^s$ for both the $\vec{x}$ and $\vec{y}$ components over time. For a video of this simulation, see \href{https://youtube.com/shorts/aRki-Mbna3w}{https://youtube.com/shorts/aRki-Mbna3w}. To view our simulation code, see \cite{butlerRepo2023}.

Note that the follower agents in this simulation end up switching positions in the formation as a consequence of the induced spring dynamics for the system. However, it should also be noted that the simplistic nature of the formation dynamics in this example may lead to restricted behavior when encountering obstacles. Therefore, replacing the mass-spring network with more sophisticated formation control protocols will induce smarter formation behavior such as better group path planning, formation coordination, etc.
The advantages of the proposed safety filtering algorithm can be leveraged as long as the additional formation control protocols are piece-wise differentiable with respect to $x_i$ for all $i \in [n]$.

\begin{figure}
    \centering
    \includegraphics[width=.8\columnwidth]{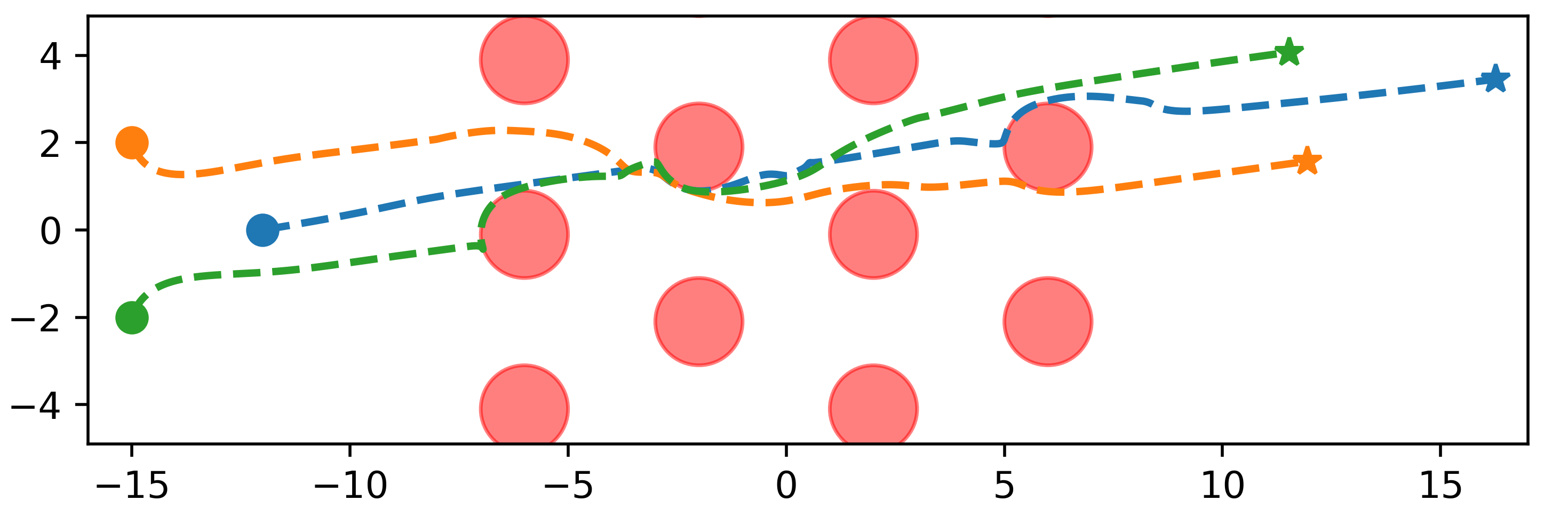}
    \caption{The trajectories of a 3-agent formation through an obstacle field, where a leader agent (blue) is given a constant control signal directing it straight through the field. Each agent implements safety filtering according to Algorithm~\ref{alg:mod_colab_safety} and \eqref{eq:problem_statement} to avoid obstacles while maintaining a formation behavior, according to \eqref{eq:formation_dynamics_mass_spring} and \eqref{eq:formation_controller_mass_spring}.}
    \label{fig:obs_field_trajectory}
\end{figure}

\begin{figure}
    \centering
    \begin{subfigure}[b]{0.48\textwidth}
         \centering
         \begin{overpic}[width=\textwidth]{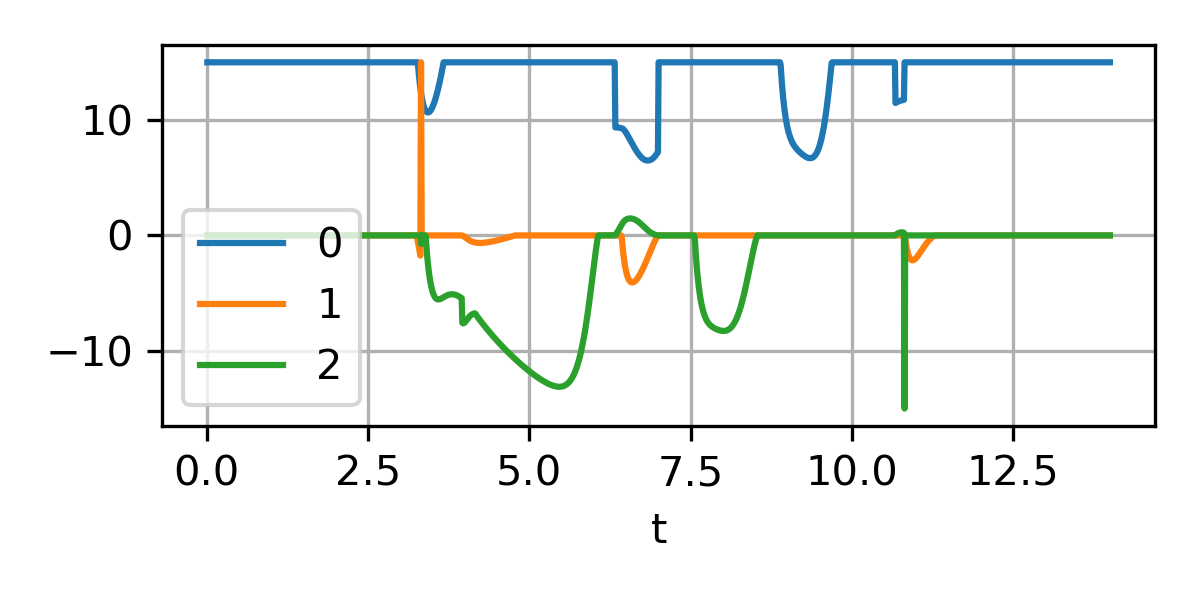}
             \put(-2,29){\parbox{0.8\linewidth}\normalsize \rotatebox{90}{\Large $u_i^{s,\vec{x}}$}}
         \end{overpic} 
    \end{subfigure}
    \hfill
    \begin{subfigure}[b]{0.48\textwidth}
         \centering
         \begin{overpic}[width=\textwidth]{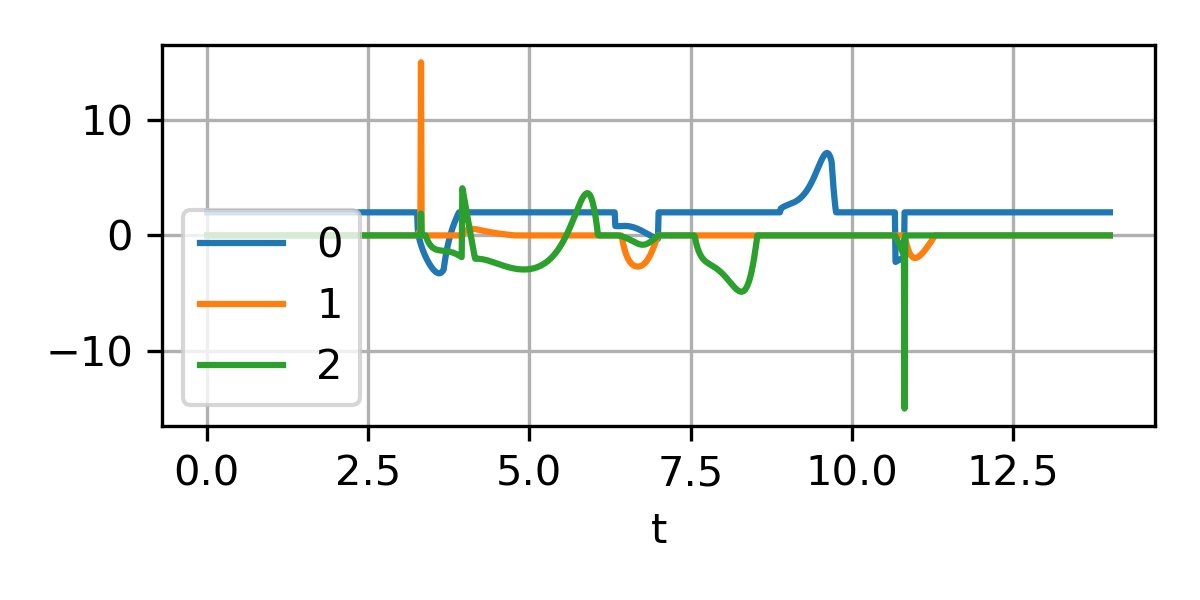}
             \put(-2,29){\parbox{0.8\linewidth}\normalsize \rotatebox{90}{\Large $u_i^{s,\vec{y}}$}}
         \end{overpic}
    \end{subfigure}
    \caption{The safety-filtered control signals for each agent in the $\vec{x}$ component (left) and $\vec{y}$ component (right) of $u_i^s$, which are computed using Algorithm~\ref{alg:mod_colab_safety} and \eqref{eq:problem_statement}, during the traversal of the formation through the obstacle field shown in Figure~\ref{fig:obs_field_trajectory}. Note that a constant control signal is given to agent $0$ (blue), which is included in the modified control signal.}
    \label{fig:control_safety_filter}
\end{figure}

\section{Conclusions}

In this paper, we have presented a method for applying safety-filtered control to arbitrarily distributed formation control algorithms through active communication of safety needs between neighboring agents in formation. We have modified a collaborative safety algorithm from our previous work \cite{butler2023distributed} to account for the communication and processing of multiple safety conditions and shown that, if the algorithm is convergent for all agents, then the formation is guaranteed to remain safe. 
Directions for future work include an analysis of the convergence for the modified collaborative safety algorithm under conflicting safety requests, as well as incorporating robustness to sources of uncertainty in our conditions for safety guarantees. 
Further, it should be noted that we make no assumptions about the real-time computation and communication of safety requests between neighbors and that the computation load for each agent increases as more obstacles are added to the environment, including other neighboring agents. Therefore, to bring this safety-filtering algorithm to real-time applications, in future work we must consider several real-world challenges in the implementation of active collaboration between communicating agents.

\bibliographystyle{unsrtnat}
\bibliography{references}  






\end{document}